\documentclass[12pt, reqno]{amsart}
\usepackage{amsmath, amsthm, amscd, amsfonts, amssymb, graphicx, color}
\usepackage[bookmarksnumbered, colorlinks, plainpages]{hyperref}
\hypersetup{colorlinks=true,linkcolor=red, anchorcolor=green, citecolor=cyan, urlcolor=red, filecolor=magenta, pdftoolbar=true}

\newtheorem{theorem}{Theorem}[section]
\newtheorem{lemma}[theorem]{Lemma}
\newtheorem{proposition}[theorem]{Proposition}
\newtheorem{corollary}[theorem]{Corollary}
\theoremstyle{definition}
\newtheorem{definition}[theorem]{Definition}
\newtheorem{example}[theorem]{Example}

\theoremstyle{remark}

\numberwithin{equation}{section}

\begin{document}

\setcounter{page}{1}

\title[Integral $K$-Operator Frames for $\mathcal{B}(H)$]{Integral $K$-Operator Frames for $\mathcal{B}(H)$}

\author[ H. LABRIGUI, M. ROSSAFI, A. Touri and S. KABBAJ]{H. LABRIGUI$^1$$^{*}$, M. ROSSAFI$^1$, A. TOURI$^1$ \MakeLowercase{and} S. KABBAJ$^1$}

\address{$^{1}$Department of Mathematics, University of Ibn Tofail, B.P. 133, Kenitra, Morocco}
\email{\textcolor[rgb]{0.00,0.00,0.84}{hlabrigui75@gmail.com, rossafimohamed@gmail.com, touri.abdo68@gmail.com, Samkabbaj@yahoo.fr}}



\date{
\newline \indent $^{*}$Corresponding author}
\maketitle
	
	\begin{abstract}
		In this paper, we will introduce a new notion, that of $K$-Integral operator frames in the set of all bounded linear operators noted $\mathcal{B}(H)$, where $H$ is a separable Hilbert space. Also, we prove some results of integral $K$-operator frame. Lastly we will establish some new properties for the perturbation and stability for an integral $K$-operator frames for $\mathcal{B}(H)$. 
	\end{abstract}
	\section{Introduction and preliminary}
	
	The notion of frames in Hilbert spaces has been introduced by Duffin and Schaeffer \cite{8} in 1952 to study some deep problems in nonharmonic Fourier
	series, after the fundamental paper \cite{10} by Daubechies, Grossman and Meyer, frames theory began to be widely used, particularly in the more specialized context of wavelet frames and Gabor frames \cite{9}. The theory of frames has been applied to signal processing, image processing, data compressing and so on.\\
	This last decade have seen tremendous activity in the development of frame theory and many generalizations of frames have come into existence. Theory of frames have been extended from Hilbert spaces to Hilbert  $C^{\ast}$-modules \cite{r6, rk11,r2, r3,r8,r4, r5,r7, r9, Ch, kr}.\\
	Some generalizations of frames for Hilbert spaces have been introduced and studied especially operator frame. The concept of operator frame for  the space $\mathcal{B}(H)$ of all bounded linear operators on Hilbert space $H$ was introduced by Chun-Yan Li and Huai-Xin Cao \cite{1}.\\
	Motivated by the work of Chander Shekhar and Shiv Kumar Kaushik \cite{01}, we introduce and study a new notion that of " Integral $K$-operator frame for $\mathcal{B}(H)$". Also we give some new properties.\\
	In what follows, we set H a separable Hilbert space and $\mathcal{B}(H)$ the set of all bounded linear operators from H to H and $(\Omega,\mu)$ is a measure space with positive measure $\mu$.\\
	Let $K,T \in \mathcal{B}(H) $, if $TK=I$, then $T$ is called the left inverse of $K$, denoted by $K_{l}^{-1}$.\\
	If $KT=I$, then $T$ is called the right inverse of $K$ and we write $K_{r}^{-1}=T$.\\
	If $KT=TK=I$, then $T$ and $K$ are inverse of each other.\\
	For a separable Hilbert space $H$ and a measurable space $(\Omega,\mu)$, define,
	\begin{equation*}
		l^{2}(\Omega,H)=\{x_{\omega} \in H,\quad  \omega \in \Omega,\quad \left\|\int_{\Omega}\langle x_{\omega},x_{\omega}\rangle d\mu(\omega)\right\| < \infty  \}.
	\end{equation*}
	For any $x=(x_{\omega})_{\omega \in \Omega}$ and $y=(y_{\omega})_{\omega \in \Omega}$, the inner product on $l^{2}(\Omega,H)$ is defined by, 
	\begin{equation*}
		\langle x,y\rangle = \int_{\Omega}\langle x_{\omega},y_{\omega}\rangle d\mu(\omega).
	\end{equation*}
	And the norme is defined by $\|x\|=\langle x,x\rangle^{\frac{1}{2}}$.
	\begin{definition}
		A sequence $\{\alpha_{i}\}_{i\in I} \subset \mathbb{R}$ is said to be positively confined if:
		\begin{equation*}
			0<\underset{i\in I}{\inf}\alpha_{i} <\underset{i\in I}{\sup}\alpha_{i}<\infty
		\end{equation*}
		Where $I$ is a finite or countable index subset.
	\end{definition}
	\begin{theorem}\cite{66}\label{t1}
		Let $H$ be a Hilbert space and $T, K \in \mathcal{B}(H)$. Then the following statements are equivalent:
		\begin{itemize}
			\item [1.] $\mathcal{R}(K) \subseteq \mathcal{R}(T)$.
			\item [2.] $KK^{\ast} \leq \lambda^{2}TT^{\ast}$, for some $0<\lambda $.
			\item [3.] $K=TQ$, for some $Q\in \mathcal{B}(H)$.
		\end{itemize}
	\end{theorem}
	\begin{lemma} \label{33} \cite{33}.
		Let $(\Omega,\mu )$ be a measure space, $X$ and $Y$ are two Banach spaces, $\lambda : X\longrightarrow Y$ be a bounded linear operator and $f : \Omega\longrightarrow X$ measurable function; then, 
		\begin{equation*}
			\lambda (\int_{\Omega}fd\mu)=\int_{\Omega}(\lambda f)d\mu.
		\end{equation*}
	\end{lemma}

	\section{Integral $K$-operator frame}
	\begin{definition}
		Let $K \in \mathcal{B}(H) $, and $\Lambda=\{\Lambda_{\omega} \in \mathcal{B}(H), \quad \omega \in \Omega \} $, the family $\Lambda$  is said an integral $K$-operator frame for $\mathcal{B}(H)$, if there exists a positive constants $0<A, B <\infty $ such that :
		\begin{equation}\label{d1}
			A\|K^{\ast}x\|^{2}\leq \int_{\Omega}\|\Lambda_{\omega}x\|^{2}d\mu (w)\leq B\|x\|^{2} \qquad x\in H,
		\end{equation}
		where $A$ and $B$ are called lower and upper bounds for the integral $K$-operator frame $\Lambda$, respectively.\\
		An integral $K$-operator frame $\{\Lambda_{\omega}\}_{\omega \in \Omega} $ is said to be tight if there exists a constant $0<A$ such that, 
		\begin{equation}
			A\|K^{\ast}x\|^{2}= \int_{\Omega}\|\Lambda_{\omega}x\|^{2}d\mu (w)\qquad x\in H.
		\end{equation}
		It is called a parseval integral $K$-operator frame if $A=1$ in last inequality.\\
		If only upper inequality of \eqref{d1} holds, then the family  $\{\Lambda_{\omega}\}_{\omega \in \Omega} $ is called an integral $K$-operator Bessel family for $\mathcal{B}(H)$.
	\end{definition}
	\begin{example}
		Let $H$ a Hilbert space defined by:\\
		$H=\left\{ A=\left( 
		\begin{array}{ccc}
		a & 0  \\ 
		0 & b 
		\end{array}%
		\right) \text{ / }a,b\in 
		\mathbb{C}
		\right\} $, \\
		We define the inner product :
		\[
		\begin{array}{ccc}
		H\times H & \rightarrow  & \mathbb{C} \\ 
		(A,B) & \mapsto  & \langle A, B\rangle = a\bar{a_{1}} + b\bar{b_{1}}%
		\end{array}%
		\]\\
		where 
		$ A=\left( 
		\begin{array}{ccc}
		a & 0  \\ 
		0 & b 
		\end{array}%
		\right)$ and 
		$B=\left( 
		\begin{array}{ccc}
		a_{1} & 0  \\ 
		0 & b_{1} 
		\end{array}%
		\right)$\\
		So, we have : $	\|A\|=\sqrt{|a|^{2}+ |b|^{2}} $\\
		Now, we consider a measure space $(\Omega=\left[ 0,1\right] ,d\lambda ) $, whose $d\lambda  $ is a Lebesgue measure  restraint on the interval $\left[ 0,1\right] $.\\
		For all $w \in \left[ 0,1\right] $, we define :
		\begin{align*}
			\Lambda_{\omega} : H &\longrightarrow H\\
			A&\longrightarrow \left(   \begin{array}{ccc}
				w & 0  \\ 
				0 & \frac{w}{2} 
			\end{array}
			\right)A
		\end{align*}
		It is clear that the family $\{\Lambda_{\omega}\}_{w\in \left[ 0,1\right]}$ be a family of continuous operator on $H$.\\
		Moreover, we have :
		\begin{align*}
			\|\Lambda_{\omega}(A)\|^{2}&=\langle \Lambda_{\omega}A, \Lambda_{\omega}A\rangle\\
			&=w^{2}|a|^{2} + \frac{w^{2}}{4}|b|^{2}
		\end{align*}
		Hence,
		\begin{equation*}
			\int_{\Omega}\|\Lambda_{\omega}(A)\|^{2}d\lambda(w)=|a|^{2}\int_{\Omega}w^{2}d\lambda(w) + |b|^{2}\int_{\Omega}\frac{w^{2}}{4}d\lambda(w).
		\end{equation*}
		For all $\lambda \in \left[ 0,1\right] $, we define :
		\begin{align*}
			K : H &\longrightarrow H\\
			A&\longrightarrow \left(   \begin{array}{ccc}
				\frac{\lambda}{2}  & 0  \\ 
				0 & \frac{\lambda}{4} 
			\end{array}
			\right)A
		\end{align*}
		So,
		\begin{equation*}
			\|K^{\ast}A\|^{2}=\frac{|\lambda|^{2}}{4}|a|^{2} + \frac{|\lambda|^{2}}{16}|b|^{2}
		\end{equation*} 
		We conclude that ,
		\begin{equation*}
			\|K^{\ast}A\|^{2}\leq \int_{\Omega}\|\Lambda_{w}(A)\|^{2}d\lambda(w) \leq \frac{1}{3} \|A\|^{2}
		\end{equation*}
		Which shows that  $\{\Lambda_{w}\}_{w\in \left[ 0,1\right]}$ is a integral $K$-operator frame for $\mathcal{B}(H)$.
	\end{example}
	Let $\{\Lambda_{\omega}\}_{\omega \in \Omega} $ be an integral $K$-operator frame for $\mathcal{B}(H) $, we define an operator: 
	\begin{align*}
		R  :  H  &\longrightarrow  l^{2}(\Omega,H) \\
		x  &\longrightarrow  R(x)=\{\Lambda_{\omega}x\}_{\omega \in \Omega}.\\
	\end{align*}
	It easy to show that the $R$ is a linear and bounded operator, called the analysis operator of the integral $K$-operator frame $\{\Lambda_{\omega}\}_{\omega \in \Omega} $.\\
	The adjoint of the analysis operator $R$ is defined by :
	\begin{align*}
		R^{\ast}  :  l^{2}(\Omega,H)  &\longrightarrow  H \\
		\{x_{\omega}\}_{\omega \in \Omega}  &\longrightarrow  R^{\ast}(x)=\int_{\Omega}\Lambda^{\ast}_{\omega}x_{\omega}d\mu (\omega).\\
	\end{align*}
	By composing $R$ and $R^{\ast}$, the frame operator $S$ for the integral $K$-operator frame $\{\Lambda_{\omega}\}_{\omega \in \Omega}$ is given by :
	\begin{equation*}
		Sx=R^{\ast} Rx=\int_{\Omega}\Lambda^{\ast}_{\omega}\Lambda_{\omega}xd\mu (\omega).
	\end{equation*}
	Clearly to see that the frame operator $S$ is positive, sefadjoint and bounded.
	\begin{proposition}
		Let $K \in \mathcal{B}(H) $ and $\{\Lambda_{\omega}\}_{\omega \in \Omega}$ be an integral $K$-operator frame for $\mathcal{B}(H) $ with frames bounds $A$ and $B$, then $\{\Lambda_{\omega}\}_{\omega \in \Omega}$ is an integral operator frame for $\mathcal{B}(H)$ if $K$ is surjective.
	\end{proposition}
	\begin{proof}
		Let $\{\Lambda_{\omega}\}_{\omega \in \Omega}$ be an integral $K$-operator frame for $\mathcal{B}(H)$, then 
		\begin{equation}
			A\|K^{\ast}x\|^{2}\leq \int_{\Omega}\|\Lambda_{\omega}x\|^{2}d\mu (w)\leq B\|x\|^{2} \qquad x\in H,
		\end{equation}
		Since $K$ is surjective, then there exist $\alpha > 0$ such that 
		\begin{equation}
			\alpha \|x\|\leq \|K^{\ast}x\|, \qquad x\in H.
		\end{equation} 
		So,
		\begin{equation}
			\alpha^{2} A\|x\|^{2}\leq \int_{\Omega}\|\Lambda_{\omega}x\|^{2}d\mu (w)\leq B\|x\|^{2} \qquad x\in H.
		\end{equation} 
		Which shows that  $\{\Lambda_{\omega}\}_{\omega \in \Omega}$ is an integral operator frame for $\mathcal{B}(H)$ with frames bounds $\alpha^{2}A$ and $B$.
	\end{proof}
	\begin{theorem}
		Let $\{\Lambda_{\omega}\}_{\omega \in \Omega}$ be an integral $K$-operator Bessel family for $\mathcal{B}(H)$, the following statements are equivalent :
		\begin{itemize}
			\item [1-] $\{\Lambda_{\omega}\}_{\omega \in \Omega}$ is an integral $K$-operator frame for $\mathcal{B}(H)$.
			\item [2-] There exist $0<A$ such that $AKK^{\ast}\leq S$, where $S$ is the frame operator for $\{\Lambda_{\omega}\}_{\omega \in \Omega}$.
			\item [3-] $K=S^{\frac{1}{2}}L$, for some $L \in \mathcal{B}(H)$.
		\end{itemize}
	\end{theorem}
	\begin{proof}
		$(1) \Longrightarrow (2)$ Let $\Lambda =\{\Lambda_{\omega}\}_{\omega \in \Omega}$ be an integral $K$-operator frame for $\mathcal{B}(H)$, with bounds $A$ and $B$, then:
		\begin{equation}\label{PP1}
			A\|K^{\ast}x\|^{2}\leq \int_{\Omega}\|\Lambda_{\omega}x\|^{2}d\mu (w)\leq B\|x\|^{2}. \qquad x\in H,
		\end{equation}
		Let $S$ the frame operator for $\Lambda$, by lemma \ref{33}, we have :
		\begin{align*}
			\langle Sx,x\rangle &= \langle \int_{\Omega}\Lambda^{\ast}_{\omega}\Lambda_{\omega}x_{\omega}d\mu (\omega),x\rangle\\
			&=\int_{\Omega} \langle \Lambda^{\ast}_{\omega}\Lambda_{\omega}x_{\omega},x\rangle d\mu (\omega)\\
			&=\int_{\Omega} \langle \Lambda_{\omega}x_{\omega},\Lambda_{\omega}x\rangle d\mu (\omega)\\
			&=\int_{\Omega} \|\Lambda_{\omega}x\|^{2} d\mu (\omega).
		\end{align*}
		From \eqref{PP1}, we obtaint,
		\begin{equation}
			\langle AKK^{\ast}x,x\rangle \leq \langle Sx,x\rangle\leq  \langle Bx,x\rangle  \qquad x\in H.
		\end{equation}
		Then : $AKK^{\ast} \leq S$.\\
		
		$(2) \Longrightarrow (3)$ Suppose there exist $0<A$ such that $AKK^{\ast} \leq S$, then we have $AKK^{\ast}\leq  S^{\frac{1}{2}}S^{\frac{1}{2}}$.\\
		Hence,
		\begin{equation*}
			\langle AKK^{\ast}x,x\rangle =A\langle K^{\ast}x,K^{\ast}x\rangle \leq  \langle S^{\frac{1}{2}}S^{\frac{1}{2}}x,x\rangle =  \langle S^{\frac{1}{2}}x,S^{\frac{1}{2}}x\rangle\qquad x\in H.
		\end{equation*}
		So, 
		\begin{equation*}
			\|K^{\ast}x\|^{2} \leq A^{-1}\|S^{\frac{1}{2}}x\|^{2} \qquad x\in H.
		\end{equation*}
		By theorem \ref{t1} we have :\\
		$K=S^{\frac{1}{2}}L$ for some $L\in \mathcal{B}(H)$.\\
		$(3) \Longrightarrow (1)$ Let $K=S^{\frac{1}{2}}L$ for some $L\in \mathcal{B}(H)$.\\
		By theoreme \ref{t1}, there exist $0<A$ such that :
		\begin{equation*}
			\|K^{\ast}x\| \leq A\|S^{\frac{1}{2}}x\| \qquad x\in H.
		\end{equation*}
		So, on one hand we have:
		\begin{equation*}
			A^{-2}\|K^{\ast}x\|^{2}\leq \|S^{\frac{1}{2}}x\|^{2}=\langle S^{\frac{1}{2}}x,S^{\frac{1}{2}}x\rangle=\langle Sx,x\rangle=\int_{\Omega}\|\Lambda_{\omega}x\|^{2}d\mu (\omega).
		\end{equation*}
		On other hand, we assumed that $\{\Lambda_{\omega}\}_{\omega \in \Omega}$ be an integral $K$-operator Bessel family for $\mathcal{B}(H)$, which ends the proof.
	\end{proof}
	\begin{theorem}
		Let $K,L \in \mathcal{B}(H)$ and $\{\Lambda_{\omega}\}_{\omega \in \Omega}$ be an integral $K$-operator frame for $\mathcal{B}(H)$, then $\{\Lambda_{\omega}L\}_{\omega \in \Omega}$ is an integral $L^{\ast}K$-operator frame for $\mathcal{B}(H)$.
	\end{theorem}
	\begin{proof}
		Let $\{\Lambda_{\omega}\}_{\omega \in \Omega}$ be an integral $K$-operator frame for $\mathcal{B}(H)$ with bounds $A$ and $B$, then,
		\begin{equation*}
			A\|K^{\ast}x\|^{2}\leq \int_{\Omega}\|\Lambda_{\omega}x\|^{2}d\mu (w)\leq B\|x\|^{2} \qquad x\in H.
		\end{equation*}
		Then,
		\begin{equation*}
			A\|K^{\ast}Lx\|^{2}\leq \int_{\Omega}\|\Lambda_{\omega}Lx\|^{2}d\mu (w)\leq B\|Lx\|^{2} \qquad x\in H.
		\end{equation*}
		So, 
		\begin{equation*}
			A\|(L^{\ast}K)^{\ast}x\|^{2}\leq \int_{\Omega}\|\Lambda_{\omega}Lx\|^{2}d\mu (w)\leq B\|Lx\|^{2} \qquad x\in H.
		\end{equation*}
		Since : 
		\begin{equation*}
			B\|Lx\|^{2} \leq B\|L\|^{2}\|x\|^{2}.
		\end{equation*}
		Then $\{\Lambda_{\omega}L\}_{\omega \in \Omega}$ is an integral $L^{\ast}K$-operator frame for $\mathcal{B}(H)$ with bounds $A$ and $B\|L\|^{2}$.
	\end{proof}
	\begin{theorem}
		Let $K \in \mathcal{B}(H)$ and $\{\Lambda_{\omega}\}_{\omega \in \Omega}$ be a tight integral $K$-operator frame for $\mathcal{B}(H)$, with bound $0<A$. Then $\{\Lambda_{\omega}\}_{\omega \in \Omega}$ is a tight integral operator frame for $\mathcal{B}(H)$ with frame bound $B$ if and only if $K^{-1}_{r}=\frac{A}{B}K^{\ast}$.
	\end{theorem}
	\begin{proof}
		Let $\{\Lambda_{\omega}\}_{\omega\in \Omega}$ be a tight integral operator frame for $\in \mathcal{B}(H)$ with bound $0<B$, then :
		\begin{equation*}
			\int_{\Omega}\|\Lambda_{\omega}x\|^{2}d\mu(\omega)=B\|x\|^{2}.
		\end{equation*}
		So, for each $x\in H$, we have :
		\begin{equation*}
			A\|K^{\ast}x\|^{2}=B\|x\|^{2}.
		\end{equation*}
		Moreover;
		\begin{equation*}
			\langle KK^{\ast}x,x\rangle = \langle  \frac{B}{A}x,x\rangle \qquad x\in H.
		\end{equation*}
		Hence, : $K^{-1}_{r}=\frac{A}{B}K^{\ast}$.\\
		Conversely, Suppose that $K^{-1}_{r}=\frac{A}{B}K^{\ast}$. Then, $KK^{\ast}=\frac{B}{A}I$.\\
		So, 
		\begin{equation*}
			\langle KK^{\ast}x,x\rangle = \langle \frac{B}{A}x,x\rangle \qquad x\in H.
		\end{equation*}
		Since $ \{\Lambda_{\omega}\}_{\omega\in \Omega}$ is a tight integral $K$-operator frame for $ \mathcal{B}(H)$ with bound $B$, we have:
		\begin{equation*}
			\int_{\Omega}\|\Lambda_{\omega}x\|^{2}d\mu(\omega)=B\|x\|^{2}.
		\end{equation*}
		Which ends the proof.
	\end{proof}
	\begin{theorem}
		Let $L, K \in \mathcal{B}(H)$, if  $ \{\Lambda_{\omega}\}_{\omega\in \Omega}$  is an integral $K$-operator frames and integral $L$-operator frame for $ \mathcal{B}(H)$ and $a,b$ are non nul scalars, then $ \{\Lambda_{\omega}\}_{\omega\in \Omega}$ is an integral $(aK+bL)
		$-operator frames and an integral $KL$-operator frame for $\mathcal{B}(H)$.
	\end{theorem}
	\begin{proof}
		Let  $ \{\Lambda_{\omega}\}_{\omega\in \Omega}$  be an integral $K$-operator frames and integral $L$-operator frame for $ \mathcal{B}(H)$, then there exists $0< A, B<\infty $ and $0< C,
		D<\infty $, such that, for all $x\in H$, we have :
		\begin{equation}\label{1}
			A\|K^{\ast}x\|^{2}\leq \int_{\Omega}\|\Lambda_{\omega}x\|^{2}d\mu (w)\leq B\|x\|^{2},
		\end{equation}
		\begin{equation}\label{2}
			C\|L^{\ast}x\|^{2}\leq \int_{\Omega}\|\Lambda_{\omega}x\|^{2}d\mu (w)\leq D\|x\|^{2},
		\end{equation}
		Then,
		\begin{align*}
			\sqrt{AC}\|(aK+bL)^{\ast}x\|&\leq \sqrt{AC}\|aK^{\ast}x\| + \sqrt{AC}\|bL^{\ast}x\|\\
			&\leq \sqrt{C}|a| (\int_{\Omega}\|\Lambda_{\omega}x\|^{2}d\mu (w))^{\frac{1}{2}} + \sqrt{A}|b| (\int_{\Omega}\|\Lambda_{\omega}x\|^{2}d\mu (w))^{\frac{1}{2}}\\
			&\leq (\sqrt{C}|a| + \sqrt{A}|b|) (\int_{\Omega}\|\Lambda_{\omega}x\|^{2}d\mu (w))^{\frac{1}{2}}\\
		\end{align*}
		So,
		\begin{equation}\label{eq11}
			AC\|(aK+bL)^{\ast}x\|^{2} \leq (\sqrt{C}|a| + \sqrt{A}|b|)^{2} (\int_{\Omega}\|\Lambda_{\omega}x\|^{2}d\mu (w))\\
		\end{equation}
		From \eqref {1}, \eqref {2} and \eqref {eq11}, we obtaint,
		\begin{equation*}
			\frac{AC}{(\sqrt{C}|a| + \sqrt{A}|b|)^{2}}\|(aK+bL)^{\ast}x\|^{2}\leq \int_{\Omega}\|\Lambda_{\omega}x\|^{2}d\mu (w)\leq\frac{B+D}{2}\|x\|^{2}
		\end{equation*}
		Wich show that $ \{\Lambda_{\omega}\}_{\omega\in \Omega}$ is an integral $(aK+bL)$-operator frame for $\mathcal{B}(H)$.\\
		Also, for all $x\in H$, we have :
		\begin{equation}\label{3}
			\|(KL)^{\ast}x\|^{2}=\|L^{\ast}K^{\ast}x\|^{2}\leq \|L^{\ast}\|^{2}\|K^{\ast}x\|^{2}.
		\end{equation}
		By \eqref{2} and \eqref{3}, we conclude:
		\begin{equation*}
			\frac{A}{\|L^{\ast}\|^{2}}\|(KL)^{\ast}x\|^{2}\leq \int_{\Omega}\|\Lambda_{\omega}x\|^{2}d\mu (w)\leq B\|x\|^{2}.
		\end{equation*}
		Which shows that  $ \{\Lambda_{\omega}\}_{\omega\in \Omega}$ is an integral $KL$-operator frame for $\mathcal{B}(H)$.
	\end{proof}
	\begin{corollary}
		Let $K \in \mathcal{B}(H)$ and $\{\Lambda_{\omega}\}_{\omega \in \Omega}$ be an integral $K$-operator frame for $\mathcal{B}(H)$, then $\{\Lambda_{\omega}\}_{\omega \in \Omega}$ is an integral $L$-operator frame for any operator $L$ in the subalgebra generated by $K$.
	\end{corollary}
	\begin{proof}
		Follows in view of last theorem.
	\end{proof}
	\begin{theorem}
		Let $K \in \mathcal{B}(H)$ and $\{\Lambda_{\omega}\}_{\omega \in \Omega}$ be an integral $K$-operator frame for $\mathcal{B}(H)$, with best frames bounds $A$ and $B$ (the maximum and minimum bounds respectively verify \eqref{d1}). If $Q:H\longrightarrow H$ is a linear homeomorphisme such that $Q^{-1}$ commutes with $K^{\ast}$, then $\{\Lambda_{\omega}Q\}_{\omega \in \Omega}$ is an integral $K$-operator frame for $\mathcal{B}(H)$ with best frames bounds $C$ and $D$ satisfying the inequalities :
		\begin{equation}\label{8}
			A\|Q^{-1}\|^{-2}\leq C\leq A\|Q\|^{2} \qquad , \qquad B\|Q^{-1}\|^{-2}\leq D\leq B\|Q\|^{2}.
		\end{equation}
	\end{theorem}
	\begin{proof}
		Let $K \in \mathcal{B}(H)$ and $\{\Lambda_{\omega}\}_{\omega \in \Omega}$ be an integral $K$-operator frame for $\mathcal{B}(H)$, with best frames bounds $A$ and $B$, then,
		\begin{equation*}
			A\|K^{\ast}x\|^{2}\leq \int_{\Omega}\|\Lambda_{\omega}x\|^{2}d\mu (w)\leq B\|x\|^{2} \qquad x\in H,
		\end{equation*}
		Also, we have for all $x\in H$,
		\begin{equation}\label{4}
			A\|K^{\ast}x\|^{2}=A\|K^{\ast}Q^{-1}Qx\|^{2}=A\|Q^{-1}K^{\ast}Qx\|^{2}\leq \|Q^{-1}\|^{2}\int_{\Omega}\|\Lambda_{\omega}Qx\|^{2}d\mu (w)
		\end{equation}
		By \eqref{4} and \eqref{d1}, we have,
		\begin{equation*}
			A\|Q^{-1}\|^{-2}\|K^{\ast}x\|^{2}\leq \int_{\Omega}\|\Lambda_{\omega}Qx\|^{2}d\mu (w)\leq B\|Q\|^{2}\|x\|^{2} \qquad x\in H.
		\end{equation*}
		Hense, $\{\Lambda_{\omega}Q\}_{\omega \in \Omega}$ is an integral $K$-operator frame for $\mathcal{B}(H)$  with bounds $A\|Q^{-1}\|^{-2}$ and $B\|Q\|^{2}$.\\
		Now, let $C$ and $D$ the best bounds of the integral $K$-operator frame $\{\Lambda_{\omega}Q\}_{\omega \in \Omega}$, then, 
		\begin{equation}\label{6}
			A\|Q^{-1}\|^{-2}\leq C \qquad and \qquad D\leq B\|Q\|^{2}
		\end{equation}
		Moreover, we have, 
		\begin{equation*}
			\|K^{\ast}x\|^{2}=\|QQ^{-1}K^{\ast}x\|^{2}\leq \|Q\|^{2}\|K^{\ast}Q^{-1}x\|^{2}, \qquad x\in H.
		\end{equation*}
		Hence, 
		\begin{align*}
			C\|Q\|^{-2}\|K^{\ast}x\|^{2}&\leq C\|K^{\ast}Q^{-1}x\|^{2}\\
			&\leq \int_{\Omega}\|\Lambda_{\omega}QQ^{-1}x\|^{2}d\mu(\omega)=\int_{\Omega}\|\Lambda_{\omega}x\|^{2}d\mu(w)\\
			&\leq D\|Q^{-1}\|^{2}\|x\|^{2}.
		\end{align*}
		Since $A$ and $B$ are the best bounds of an integral $K$-operator frame $\{\Lambda_{\omega}\}_{\omega \in \Omega}$, then we have,
		\begin{equation}\label{7}
			C\|Q\|^{-2} \leq A \qquad B\leq D\|Q^{-1}\|^{2}.
		\end{equation}
		From \eqref{6} and \eqref{7}, we obtaint \eqref{8} 
	\end{proof}
	\section{Perturbation of Integral $K$-Operator Frames}
	In this section we consider perturbation of an integral $K$-operator frames by non-zero operators. 
	\begin{theorem}
		Let $K \in \mathcal{B}(H)$ and $\{\Lambda_{\omega}\}_{\omega \in \Omega}$ be an integral $K$-operator frame for $\mathcal{B}(H)$ with frames bounds $A$ and $B$. Let $ L \in \mathcal{B}(H), (L\neq 0)$, and $\{a_{\omega}\}_{\omega \in \Omega}$ any family of scalars. Then the perturbed family of operator $\{\Lambda_{\omega} + a_{\omega}LK^{\ast}\}_{\omega \in \Omega}$ is an integral $K$-operator frames for $\mathcal{B}(H)$ if $\int_{\Omega}|a_{\omega}|^{2}d\mu(\omega) < \frac{A}{\|L\|}$.
	\end{theorem}
	\begin{proof}
		Let $\Gamma_{\omega}= \Lambda_{\omega} + a_{\omega}LK^{\ast}$,  for all $\omega \in \Omega$. Then for all $x\in H$, we have,
		\begin{align*}
			\int_{\Omega}\|\Lambda_{\omega}x-\Gamma_{\omega}x\|^{2}d\mu(\omega)&=\int_{\Omega}\|a_{\omega}LK^{\ast}x\|^{2}d\mu(\omega),\\
			&\leq \int_{\Omega}|a_{\omega}|^{2}\|L\|^{2}\|K^{\ast}\|^{2}\|x\|^{2}d\mu(\omega)\\
			&=\int_{\Omega}|a_{\omega}|^{2}\|L\|^{2}d\mu(\omega)\|K^{\ast}\|^{2}\|x\|^{2},\\
			&\leq R\|K^{\ast}\|^{2}\|x\|^{2}. \qquad where \qquad  R=\int_{\Omega}|a_{\omega}|^{2}\|L\|^{2}d\mu(\omega).
		\end{align*}
		On one hand, for all $x\in H$,we have, 
		\begin{align*}
			(\int_{\Omega}\|(\Lambda_{\omega} + a_{\omega}LK^{\ast})x\|^{2}d\mu(\omega))^{\frac{1}{2}}&=\|(\Lambda_{\omega} + a_{\omega}LK^{\ast})x\|_{\l^{2}(\Omega,H)}\\
			&\leq \|\Lambda_{\omega}x\|_{\l^{2}(\Omega,H)} + \| a_{\omega}LK^{\ast}x\|_{\l^{2}(\Omega,H)}\\
			&\leq (\int_{\Omega}\|\Lambda_{\omega}x\|^{2}d\mu(\omega))^{\frac{1}{2}} + (\int_{\Omega}\|a_{\omega}LK^{\ast}x\|^{2}d\mu(\omega))^{\frac{1}{2}}\\
			&\leq \sqrt{B}\|x\| +  \sqrt{R}\|K^{\ast}\|\|x\|\\
			&\leq (\sqrt{B} + \sqrt{R}\|K^{\ast}\|)\|x\|.
		\end{align*}
		Then 
		\begin{equation}\label{123}
			\int_{\Omega}\|(\Lambda_{\omega} + a_{\omega}L)x\|^{2}d\mu(\omega)\leq (\sqrt{B} + \sqrt{R}\|K^{\ast}\|)^{2}\|x\|^{2}.
		\end{equation}
		On other hand, for all $x\in H$, we have,
		\begin{align*}
			(\int_{\Omega}\|(\Lambda_{\omega} + a_{\omega}LK^{\ast})x\|^{2}d\mu(\omega))^{\frac{1}{2}}&=\|(\Lambda_{\omega} + a_{\omega}LK^{\ast})x\|_{\l^{2}(\Omega,H)}\\
			&\geq \|\Lambda_{\omega}x\|_{\l^{2}(\Omega,H)} - \| a_{\omega}LK^{\ast}x\|_{\l^{2}(\Omega,H)}\\
			&\geq (\int_{\Omega}\|\Lambda_{\omega}x\|^{2}d\mu(\omega))^{\frac{1}{2}} - (\int_{\Omega}\|a_{\omega}LK^{\ast}x\|^{2}d\mu(\omega))^{\frac{1}{2}}\\
			&\geq \sqrt{A}\|K^{\ast}x\| -  \sqrt{R}\|K^{\ast}x\|\\
			&\geq (\sqrt{A} -  \sqrt{R})\|K^{\ast}x\|
		\end{align*}
		So,
		\begin{equation}\label{124}
			\int_{\Omega}\|(\Lambda_{\omega} + a_{\omega}LK^{\ast})x\|^{2}d\mu(\omega) \geq  (\sqrt{A} -  \sqrt{R})^{2}\|K^{\ast}x\|^{2}
		\end{equation}
		From \eqref{123} and \eqref{124} we conclude that  $\{\Lambda_{\omega} + a_{\omega}LK^{\ast}\}_{\omega \in \Omega}$ is an integral $K$-operator frame for $\mathcal{B}(H)$ if $R<A$, that is , if :
		\begin{equation*}
			\int_{\Omega}|a_{\omega}|^{2}d\mu(\omega) < \frac{A}{\|L\|}.
		\end{equation*}
		
	\end{proof}
	\begin{theorem}
		Let $K \in \mathcal{B}(H)$ and $\{\Lambda_{\omega}\}_{\omega \in \Omega}$ be an integral $K$-operator frame for $\mathcal{B}(H)$, $\{\Gamma_{\omega}\}_{\omega \in \Omega}$ be any family on $\mathcal{ B}(H)$, and let $\{a_{\omega}\}_{\omega \in \Omega},\, \{b_{\omega}\}_{\omega \in \Omega} \subset \mathbb{R}$ be two positively  confined sequences. If there exists a constants $\alpha , \beta$  with $0\leq \alpha , \beta<\frac{1}{2}$ such that,
		\begin{equation}\label{10}
			\int_{\Omega}\|a_{\omega}\Lambda_{\omega}x - b_{\omega}\Gamma_{\omega}x\|^{2}d\mu(\omega)\leq \alpha\int_{\Omega}\|a_{\omega}\Lambda_{\omega}x\|^{2}d\mu(\omega) + \beta\int_{\Omega}\|b_{\omega}\Gamma_{\omega}x\|^{2}d\mu(\omega).
		\end{equation}
		Then $\{\Gamma_{\omega}\}_{\omega \in \Omega}$ is an integral $K$-operator frame for $\mathcal{B}(H)$.
	\end{theorem}
	\begin{proof}
		Suppose \eqref{10} holds for some conditions of theorem.\\
		Then for all $x\in H$ we have,
		\begin{align*}
			\int_{\Omega}\|b_{\omega}\Gamma_{\omega}x\|^{2}d\mu(\omega)&\leq 2(\int_{\Omega}\|a_{\omega}\Lambda_{\omega}x\|^{2}d\mu(\omega) +\int_{\Omega}\|a_{\omega}\Lambda_{\omega}x - b_{\omega}\Gamma_{\omega}x\|^{2}d\mu(\omega) ) \\
			&\leq 2(\int_{\Omega}\|a_{\omega}\Lambda_{\omega}x\|^{2}d\mu(\omega) +\alpha\int_{\Omega}\|a_{\omega}\Lambda_{\omega}x\|^{2}d\mu(\omega) + \beta\int_{\Omega}\|b_{\omega}\Gamma_{\omega}x\|^{2}d\mu(\omega)).
		\end{align*}
		Therefore, 
		\begin{equation*}
			(1-2\beta)\int_{\Omega}\|b_{\omega}\Gamma_{\omega}x\|^{2}d\mu(\omega)\leq 2(1+\alpha)\int_{\Omega}\|a_{\omega}\Lambda_{\omega}x\|^{2}d\mu(\omega).
		\end{equation*}
		This give,
		\begin{equation*}
			(1-2\beta)[\underset{\omega \in \Omega}{\inf}( b_{\omega})]^{2}\int_{\Omega}\|\Gamma_{\omega}x\|^{2}d\mu(\omega)\leq 2(1+\alpha)[\underset{\omega \in \Omega}{\sup}( a_{\omega})]^{2}\int_{\Omega}\|\Lambda_{\omega}x\|^{2}d\mu(\omega).
		\end{equation*}
		Thus,
		\begin{equation}\label{101}
			\int_{\Omega}\|\Gamma_{\omega}x\|^{2}d\mu(\omega)\leq \frac{ 2(1+\alpha)[\underset{\omega \in \Omega}{\sup}( a_{\omega})]^{2}}{(1-2\beta)[\underset{\omega \in \Omega}{\inf}( b_{\omega})]^{2}}\int_{\Omega}\|\Lambda_{\omega}x\|^{2}d\mu(\omega).
		\end{equation}
		Also, for all $ x\in H$,
		\begin{align*}
			\int_{\Omega}\|a_{\omega}\Lambda_{\omega}x\|^{2}d\mu(\omega)&\leq 2(\int_{\Omega}\|a_{\omega}\Lambda_{\omega}x - b_{\omega}\Gamma_{\omega}x\|^{2}d\mu(\omega) + \int_{\Omega}\|b_{\omega}\Gamma_{\omega}x\|^{2}d\mu(\omega))\\
			&\leq 2(\alpha\int_{\Omega}\|a_{\omega}\Lambda_{\omega}x\|^{2}d\mu(\omega) + \beta\int_{\Omega}\|b_{\omega}\Gamma_{\omega}x\|^{2}d\mu(\omega)+\int_{\Omega}\|b_{\omega}\Gamma_{\omega}x\|^{2}d\mu(\omega)).
		\end{align*}
		Therefore,
		\begin{equation*}
			(1-2\alpha)[\underset{\omega \in \Omega}{\inf}( a_{\omega})]^{2}\int_{\Omega}\|\Lambda_{\omega}x\|^{2}d\mu(\omega)\leq 2(1+\beta)[\underset{\omega \in \Omega}{\sup}( b_{\omega})]^{2}\int_{\Omega}\|\Gamma_{\omega}x\|^{2}d\mu(\omega).
		\end{equation*}
		This give:
		\begin{equation}\label{102}
			\frac{(1-2\alpha)[\underset{\omega \in \Omega}{\inf}( a_{\omega})]^{2}}{2(1+\beta)[\underset{\omega \in \Omega}{\sup}( b_{\omega})]^{2}}\int_{\Omega}\|\Lambda_{\omega}x\|^{2}d\mu(\omega)\leq \int_{\Omega}\|\Gamma_{\omega}x\|^{2}d\mu(\omega)
		\end{equation}
		From \eqref{101} and \eqref{102} we conclude,
		\begin{equation*}
			\frac{(1-2\alpha)[\underset{\omega \in \Omega}{\inf}( a_{\omega})]^{2}}{2(1+\beta)[\underset{\omega \in \Omega}{\sup}( b_{\omega})]^{2}}\int_{\Omega}\|\Lambda_{\omega}x\|^{2}d\mu(\omega)\leq \int_{\Omega}\|\Gamma_{\omega}x\|^{2}d\mu(\omega)\leq \frac{ 2(1+\alpha)[\underset{\omega \in \Omega}{\sup}( a_{\omega})]^{2}}{(1-2\beta)[\underset{\omega \in \Omega}{\inf}( b_{\omega})]^{2}}\int_{\Omega}\|\Lambda_{\omega}x\|^{2}d\mu(\omega).
		\end{equation*}
		Hence, $\{\Gamma_{\omega}\}_{\omega\in\Omega}$ is an integral $K$-operator frame for $\mathcal{B}(H)$.
	\end{proof}
	\section{Stability of Integral $K$-Operator frames}
	\begin{theorem}
		Let $K \in \mathcal{B}(H)$ and $\{\Lambda_{\omega}\}_{\omega \in \Omega}$ be an integral $K$-operator frame for $\mathcal{B}(H)$ with bounds $A$ and $B$. Let $\{\Gamma_{\omega}\}_{\omega \in \Omega} \subset \mathcal{B}(H)$ and $0\leq \alpha , \beta$.\\
		If $ 0\leq \alpha + \frac{\beta}{A}<1$ such that, \begin{equation}\label{110}
			\int_{\Omega}\|(\Lambda_{\omega} - \Gamma_{\omega})x\|^{2}d\mu(\omega)\leq \alpha\int_{\Omega}\|\Lambda_{\omega}x\|^{2}d\mu(\omega) + \beta\|K^{\ast}x\|^{2}.
		\end{equation}
		Then the family $\{\Gamma_{\omega}\}_{\omega \in \Omega} $ is an integral $K$-operator frame for $\mathcal{B}(H)$   with frame bounds $A(1-\sqrt{\alpha + \frac{\beta}{A}})^{2}$ and $B(1+\sqrt{\alpha + \frac{\beta \|K\|}{B}})^{2}$
	\end{theorem}
	\begin{proof}
		Let $K \in \mathcal{B}(H)$ and $\{\Lambda_{\omega}\}_{\omega \in \Omega}$ be an integral $K$-operator frame for $\mathcal{B}(H)$ with bounds $A$ and $B$.\\
		Then for each $x\in H$ we have,
		\begin{align*}
			\|\{\Lambda_{\omega}x\}_{\omega \in \Omega}\|_{l^{2}(H)}&\leq \|\{(\Lambda_{\omega} - \Gamma_{\omega})x \}_{\omega \in \Omega}\|_{l^{2}(H)} + \|\{\Gamma_{\omega}x\}_{\omega \in \Omega} \|_{l^{2}(H)}\\
			&\leq \sqrt{\alpha\int_{\Omega}\|\Lambda_{\omega}x\|^{2}d\mu(\omega) + \beta\|K^{\ast}x\|^{2}} + \sqrt{\int_{\Omega}\|\Gamma_{\omega}x\|^{2}d\mu(\omega) }\\
			&\leq \sqrt{\alpha\int_{\Omega}\|\Lambda_{\omega}x\|^{2}d\mu(\omega) + \frac{\beta}{A}\int_{\Omega}\|\Lambda_{\omega}x\|^{2}d\mu(\omega)}+ \sqrt{\int_{\Omega}\|\Gamma_{\omega}x\|^{2}d\mu(\omega) }\\
		\end{align*}
		Or, 
		\begin{equation*}
			A\|K^{\ast}x\|^{2}\leq \int_{\Omega}\|\Lambda_{\omega}x\|^{2}d\mu(\omega)
		\end{equation*}
		Then,
		\begin{equation}\label{111}
			A(1-\sqrt{\alpha + \frac{\beta}{A}})^{2}\|K^{\ast}x\|^{2}\leq \int_{\Omega}\|\Gamma_{\omega}x\|^{2}d\mu(\omega) 
		\end{equation}
		Also, we have:
		\begin{align*}
			\|\{\Gamma_{\omega}x\}_{\omega \in \Omega} \|_{l^{2}(H)}&\leq \|\{(\Lambda_{\omega} - \Gamma_{\omega})x \}_{\omega \in \Omega}\|_{l^{2}(H)} + \|\{\Lambda_{\omega}x\}_{\omega \in \Omega}\|_{l^{2}(H)}\\
			&\leq \sqrt{B}(\alpha + \frac{\beta\|K\|}{B})\|x\|
		\end{align*}
		So, we get :
		\begin{equation}\label{112}
			\int_{\Omega}\|\Gamma_{\omega}x\|^{2}d\mu(\omega)\leq B(\alpha + \frac{\beta\|K\|}{B})^{2}\|x\|^{2}
		\end{equation}
		From \eqref{111} and \eqref{112}, we conclude that $\{\Gamma_{\omega}\}_{\omega \in \Omega}$ is an integral $K$-operator frame for $\mathcal{B}(H)$ with frame bounds $A(1-\sqrt{\alpha + \frac{\beta}{A}})^{2}$ and $B(\alpha + \frac{\beta\|K\|}{B})^{2}$.
	\end{proof}
	\begin{corollary}
		Let $K \in \mathcal{B}(H)$ and $\{\Lambda_{\omega}\}_{\omega \in \Omega}$ be an integral $K$-operator frame for $\mathcal{B}(H)$ with bounds $A$ and $B$. Let $\{\Gamma_{\omega}\}_{\omega \in \Omega} \subset \mathcal{B}(H)$. If there is $ 0<\beta<A$ such that, 
		\begin{equation}
			\int_{\Omega}\|(\Lambda_{\omega} - \Gamma_{\omega})x\|^{2}d\mu(\omega)\leq \beta\|K^{\ast}x\|^{2}.
		\end{equation}
		Then the family $\{\Gamma_{\omega}\}_{\omega \in \Omega} $ is an integral $K$-operator frame for $\mathcal{B}(H)$   with frame bounds $A(1-\sqrt{\frac{\beta}{A}})^{2}$ and $B(1+\sqrt{\frac{\beta \|K\|}{B}})^{2}$
	\end{corollary}
	\begin{proof}
		Obvious
	\end{proof}
	\begin{theorem}
		Let $K \in \mathcal{B}(H)$ and $\{\Lambda_{\omega}\}_{\omega \in \Omega}$ be an integral $K$-operator frame for $\mathcal{B}(H)$ with bounds $A$ and $B$. Let $\{\Gamma_{\omega}\}_{\omega \in \Omega} \subset \mathcal{B}(H)$. If there exist a constant $0<M$ such that, 
		\begin{equation*}
			\int_{\Omega}\|(\Lambda_{\omega} - \Gamma_{\omega})x\|^{2}d\mu(\omega)\leq M. min (\int_{\Omega}\|\Lambda_{\omega}x\|^{2}d\mu(\omega), \int_{\Omega}\|\Gamma_{\omega}x\|^{2}d\mu(\omega)), \quad x\in H.
		\end{equation*}
		Then the family $\{\Gamma_{\omega}\}_{\omega \in \Omega} $ is an integral $K$-operator frame for $\mathcal{B}(H)$.  
	\end{theorem}
	\begin{proof}
		On one hand, we have for all $x\in H$,
		\begin{align*}
			A\|K^{\ast}x\|^{2}&\leq  \int_{\Omega}\|\Lambda_{\omega}x\|^{2}d\mu(\omega)\\
			&\leq 2(\int_{\Omega}\|(\Lambda_{\omega} - \Gamma_{\omega})x\|^{2}d\mu(\omega) + \int_{\Omega}\|\Gamma_{\omega}x\|^{2}d\mu(\omega))\\
			&\leq 2(M. min (\int_{\Omega}\|\Lambda_{\omega}x\|^{2}d\mu(\omega), \int_{\Omega}\|\Gamma_{\omega}x\|^{2}d\mu(\omega))+ \int_{\Omega}\|\Gamma_{\omega}x\|^{2}d\mu(\omega))\\
			&\leq 2(M. \int_{\Omega}\|\Gamma_{\omega}x\|^{2}d\mu(\omega)+ \int_{\Omega}\|\Gamma_{\omega}x\|^{2}d\mu(\omega))\\
			&\leq 2(M+1) \int_{\Omega}\|\Gamma_{\omega}x\|^{2}d\mu(\omega).
		\end{align*}
		On the other hand, we have,
		\begin{align*}
			\int_{\Omega}\|\Gamma_{\omega}x\|^{2}d\mu(\omega)&\leq 2(\int_{\Omega}\|(\Lambda_{\omega} - \Gamma_{\omega})x\|^{2}d\mu(\omega) + \int_{\Omega}\|\Lambda_{\omega}x\|^{2}d\mu(\omega)) \\
			&\leq 2(M+1)\|x\|^{2},
		\end{align*}
		finally we find,
		\begin{equation*}
			\frac{A}{2(M+1)}\|K^{\ast}x\|^{2}\leq \int_{\Omega}\|\Gamma_{\omega}x\|^{2}d\mu(\omega)\leq 2(M+1)\|x\|^{2}
		\end{equation*}
		which ends the proof.
	\end{proof}
	\begin{theorem}
		Let $K \in \mathcal{B}(H)$. For $k=1,....,n$, let  $\{\Lambda_{\omega,k}\}_{\omega \in \Omega} \subset \mathcal{B}(H)$ be an integral $K$-operator frames and $\{a_{k}\}^{n}_{k=1}$ be any scalars. Then $(\sum \limits_{\underset{}{k=1}}^n a_k\Lambda_{\omega,k})_{\omega \in \Omega}$ is an inetgral $K$-operator frames if there exist $0<\beta$ and some $p\in \{1,....,n\}$ such that,
		\begin{equation}\label{tf}
			\beta \int_{\Omega}\|\Lambda_{\omega,p}x\|^{2}d\mu(\omega)\leq \int_{\Omega}\|\sum \limits_{\underset{}{k=1}}^n a_k\Lambda_{\omega,k}x\|^{2}d\mu (\omega), \qquad x\in H.
		\end{equation}
	\end{theorem}	
	\begin{proof}
		For each $1\leq p\leq n$, let $A_{p}$ and $B_{p}$ the bounds of the integral $K$-operator frame $\{\Lambda_{\omega,p}\}_{\omega \in \Omega}$. Let $0<\beta$ be a constant satisfying \eqref{tf}, then, for all $x\in H$, we have,
		\begin{align*}
			A_{p}\beta \|K^{\ast}\|^{2}&\leq \beta\int_{\Omega}\|\Lambda_{\omega,p}x\|^{2}d\mu(\omega)\\
			&\leq \int_{\Omega}\|\sum \limits_{\underset{}{k=1}}^n a_k\Lambda_{\omega,k}x\|^{2}d\mu (\omega)
		\end{align*}
		Moreover, for all $x\in H$, we have,
		\begin{align*}
			\int_{\Omega}\|\sum \limits_{\underset{}{k=1}}^n a_k\Lambda_{\omega,k}x\|^{2}d\mu (\omega)&\leq \int_{\Omega}n(\sum \limits_{\underset{}{k=1}}^n \|a_k\Lambda_{\omega,k}x\|^{2})d\mu (\omega)\\
			&\leq n(\underset{1\leq k \leq n}{\max}|a_{k}|^{2})\sum \limits_{\underset{}{k=1}}^n(\int_{\Omega} \|\Lambda_{\omega,k}x\|^{2})d\mu (\omega))\\
			&\leq n(\underset{1\leq k \leq n}{\max}|a_{k}|^{2})(\sum \limits_{\underset{}{k=1}}^nB_{k})\|x\|^{2}.
		\end{align*}
		Hence, $(\sum \limits_{\underset{}{k=1}}^n a_k\Lambda_{\omega,k})_{\omega \in \Omega}$ is an inetgral $K$-operator frames for $\mathcal{B}(H)$.
	\end{proof}
	\begin{theorem}
		Let $K \in \mathcal{B}(H)$ and $\{\Lambda_{\omega,k}\}_{\omega \in \Omega} \subset \mathcal{B}(H)$ be an integral $K$-operator frame for each  $k=1,....,n$ and $\{\Gamma_{\omega,k}\}_{\omega \in \Omega} $ be any family in $\mathcal{B}(H)$ for $1\leq k\leq n$. Let $L:l^{2}(\Omega, H)\longrightarrow l^{2}(\Omega, H)$ be a bounded linear operator such that,
		\begin{equation*}
			L((\sum \limits_{\underset{}{k=1}}^n \Gamma_{\omega,k}x)_{\omega\in \Omega})=(\Lambda_{\omega,p}x)_{\omega\in \Omega}, \qquad for\; some \quad p=1,...,n.
		\end{equation*}
		If there exist a positive constant $\lambda$ such that,
		\begin{equation*}
			\int_{\Omega}\|(\Lambda_{\omega,k}-\Gamma_{\omega,k})x\|^{2}d\mu(\omega)\leq \lambda\int_{\Omega}\|\Lambda_{\omega,k}x\|^{2}d\mu(\omega), \qquad k=1,...,n. \quad x\in H.
		\end{equation*}
		Then $(\sum \limits_{\underset{}{k=1}}^n \Gamma_{\omega,k})_{\omega\in \Omega}$ is an integral $K$-operator frames for $\mathcal{B}(H)$.
	\end{theorem}
	\begin{proof}
		For any $x\in H$, we have,
		\begin{align*}
			\int_{\Omega}\|\sum \limits_{\underset{}{k=1}}^n \Gamma_{\omega,k}x\|^{2}d\mu(\omega)&\leq  2n(\int_{\Omega}\sum \limits_{\underset{}{k=1}}^n \|(\Lambda_{\omega,k}-\Gamma_{\omega,k})x\|^{2}d\mu(\omega) + \int_{\Omega}\sum \limits_{\underset{}{k=1}}^n \|\Lambda_{\omega,k}x\|^{2}d\mu(\omega))\\
			&\leq 2n\sum \limits_{\underset{}{k=1}}^n(\lambda \int_{\Omega}\| \Lambda_{\omega,k}x\|^{2}d\mu(\omega)+\int_{\Omega}\| \Lambda_{\omega,k}x\|^{2}d\mu(\omega))\\
			&\leq 2n(1+\lambda)(\sum \limits_{\underset{}{k=1}}^n B_{k})\|x\|^{2}.
		\end{align*}
		Also, for all $x\in H$, we have:
		\begin{equation*}
			\|L((\sum \limits_{\underset{}{k=1}}^n \Gamma_{\omega,k}x)_{\omega\in \Omega})\|^{2}=\int_{\Omega}\|\Lambda_{\omega,p}x\|^{2}d\mu(\omega).
		\end{equation*}
		Therefore, we get:
		\begin{equation*}
			A_{p}\|K^{\ast}x\|^{2}\leq \int_{\Omega}\|\Lambda_{\omega,p}x\|^{2}d\mu(\omega).
		\end{equation*}
		This gives,
		\begin{equation*}
			\frac{A_{p}}{\|L\|^{2}}\|K^{\ast}x\|^{2}\leq \int_{\Omega}\|\sum \limits_{\underset{}{k=1}}^n\Gamma_{\omega,k}x\|^{2}d\mu(\omega), \qquad x\in H
		\end{equation*}
		Hence, $(\sum \limits_{\underset{}{k=1}}^n \Gamma_{\omega,k})_{\omega\in \Omega}$ is an integral $K$-operator frames for $\mathcal{B}(H)$.
	\end{proof}

\end{document}